\documentclass[11pt,leqno,a4paper]{amsart}
\usepackage{amssymb,amsmath,amscd,graphicx,fontenc,amsthm,mathrsfs, mathtools}
\usepackage[frame, cmtip, arrow, matrix, line, graph, curve]{xy}
\usepackage{graphpap, color}
\usepackage[mathscr]{eucal}
\usepackage{ifthen}
\usepackage{relsize}

\usepackage{hyperref}
\usepackage[right,displaymath,mathlines]{lineno}

\usepackage{cleveref}

\usepackage{fancyhdr}
\usepackage{indentfirst}
\usepackage{paralist}
\usepackage{pstricks}
\usepackage{bbm}
\usepackage{tikz}

\newtheorem{theorem}{Theorem}[section]

\newtheorem{lemma}[theorem]{Lemma}

\newtheorem{notation}[theorem]{Notation}

\newtheorem*{proposition*}{Proposition}

\theoremstyle{remark}

\numberwithin{equation}{section}

\newcommand{\NN}{{\mathbb{N}}}
\newcommand{\ZZ}{{\mathbb{Z}}}

\newcommand{\RR}{{\mathbb{R}}}

\newcommand{\QQ}{{\mathbb{Q}}}

\DeclareMathOperator{\vol}{vol}

\newcommand{\GL}{\mathrm{GL}}
\newcommand{\SL}{\mathrm{SL}}

\newcommand{\SH}{\mathsf{H}}

\newcommand{\SU}{\mathsf{U}}

\newcommand{\sg}{\mathsf{g}}
\newcommand{\sh}{\mathsf{h}}

\newcommand{\su}{\mathsf{u}}

\newcommand{\T}{{\mathsf{T}}}

\newcommand{\ox}{\mathbf x}
\newcommand{\oy}{\mathbf y}
\newcommand{\ow}{\mathbf w}
\newcommand{\ov}{\mathbf v}

\providecommand{\op}{\mathbf{p}}
\providecommand{\oq}{\mathbf{q}}

\newcommand{\vep}{\varepsilon}

\newcommand{\NT}{\mathbf N}
\newcommand{\VT}{\mathbf V}

\newcommand{\id}{\mathrm{I}}

\renewcommand{\varpi}{\pi}

\newcommand{\origin}{O}

\newcommand{\Mat}{\mathrm{Mat}}

\makeatletter
\def\moverlay{\mathpalette\mov@rlay}
\def\mov@rlay#1#2{\leavevmode\vtop{%
   \baselineskip\z@skip \lineskiplimit-\maxdimen
   \ialign{\hfil$\m@th#1##$\hfil\cr#2\crcr}}}
\newcommand{\charfusion}[3][\mathord]{
    #1{\ifx#1\mathop\vphantom{#2}\fi
        \mathpalette\mov@rlay{#2\cr#3}
      }
    \ifx#1\mathop\expandafter\displaylimits\fi}
\makeatother

\definecolor{cmd}{rgb}{1.0, 0.35, 0.21}

\begin{document}
\title[$S$-arithmetic quantitative Khintchine-Groshev theorem]
{A Quantitative Khintchine-Groshev Theorem for $S$-arithmetic Diophantine Approximation}
\author{Jiyoung Han}

\subjclass[2020]{22F30, 11D45, 11K60, 22E35, 11J61}

\maketitle
\begin{abstract}
In \cite{Schmidt1960a}, Schmidt studied a quantitative type of Khintchine-Groshev theorem for general (higher) dimensions. 
Recently, a new proof of the theorem was found, which made it possible to relax the dimensional constraint and more generally, to add on the congruence condition \cite{AGY2021}.

In this paper, we generalize this new approach to $S$-arithmetic spaces and obtain a quantitative version of an $S$-arithmetic Khintchine-Groshev theorem.
During the process, we consider a new, but still natural $S$-arithmetic analog of Diophantine approximation, which is different from the one formerly established (see \cite{KT2007}). Hence for the sake of completeness, we also deal with the convergent case of the Khintchine-Groshev theorem, based on this new generalization.
\end{abstract}



\section{Introduction}\label{Introduction}

For any $A\in \Mat_{m,n}(\RR)$ and $T>0$, Dirichlet's approximation theorem says that there is a nontrivial integral solution $(\op, \oq)\in \ZZ^m\times \ZZ^n$ to the system of inequalities
\[
\|\oq\|<T
\quad\text{and}\quad
\|A\oq+\op\|^m<T^{-n}
\]
and as a corollary, one can find infinitely many solutions $(\op,\oq)\in \ZZ^m\times \ZZ^n$ satisfying the inequality
\[
\|A\oq+\op\|^m < \|\oq\|^{-n}.
\]

In general, we say that $A\in \Mat_{m,n}(\RR)$ is \emph{$\psi$-approximable}, where $\psi:\RR_{>0}\rightarrow \RR_{> 0}$ is non-increasing, if there are infinitely many $(\op,\oq)\in \ZZ^m\times \ZZ^n$ for which
\begin{equation}\label{KG inequality}
\|A\oq+\op\|^m < \psi(\|\oq\|^n)
\end{equation}
and Khintchine-Groshev theorem states that for \emph{almost all} (\emph{almost no}, respectively) $A\in \Mat_{m,n}(\RR)$ is $\psi$-approximable if and only if $\sum_{k\in \NN} \psi(k)=\infty$ ($<\infty$, respectively) \cite{Khintchine1924, Khintchine1926, Groshev1938}. Here, when $m=n=1$, the certain monotonicity of the function $\psi$ is necessary for the divergent case \cite{DS1941}.
The case when $m=n=1$ is the subject of the Duffin-Schaeffer conjecture (see \cite{DS1941} and also \cite{BV2010} and references therein for historical details) which was recently established in \cite{KM2020}.

When $\sum_{k\in \NN} \psi(k)$ diverges, one can quantify the Khintchine-Groshev theorem, asking for an asymptotic formula for the number $\NT_{\psi}(T)$ of $(\op, \oq)\in \ZZ^m\times \ZZ^n$ for which \Cref{KG inequality} holds with $\|\oq\|^n<T$. This problem was studied by Schmidt \cite{Schmidt1960a} for the case when $m\ge 3$ and $n=1$. See also \cite[Chapter I.5]{Sprindzuk1979} for the more general case, and \cite{LeVeque1958, Erdos1959, LeVeque1960, Szusz196263} for one-dimensional cases.
In \cite{AGY2021}, Alam, Ghosh and Yu provided a new approach for the proof of the quantitative Khintchine-Groshev theorem, based on \cite{Schmidt1960b} and \cite{GKY20} so that they extended the theorem to the case when $m+n\ge 3$ and refined the statement by adding a certain congruence condition. This congruence condition for the Khintchine-Groshev theorem was earlier considered in \cite{NRS2020}. 

The main purpose of this article is to generalize the quantitative result of \cite{AGY2021} to $S$-arithmetic spaces.
 
\subsection*{S-arithmetic Set-ups}\label{S-arithmetic Set-ups}
Let $S=\{\infty, p_1, \ldots, p_s\}$ be the finite set of places over $\QQ$, where $|\cdot|_\infty$ represents the supremum norm of $\RR$ and $|\cdot|_p$ is the $p$-adic norm for a prime $p$ (we also use $|\cdot|$ for the absolute value of $\RR\cup\{\infty\}$ to denote the difference between two quantities). Denote by $S_f=\{p_1,\ldots, p_s\}$ the set of finite places in $S$. 
For $p\in S$, let $\QQ_p$ be the completion field of $\QQ$ with respect to the norm $|\cdot|_p$. When $p=\infty$, $\QQ_p=\RR$. 
Let $\QQ_S=\prod_{p\in S} \QQ_p$ and we call $\QQ_S^d=\prod_{p\in S} \QQ_p^d$ for $d\in \NN$, \emph{an $S$-arithmetic space}. If we denote an element of $\QQ_S^d$ by $\oy=(\oy_p)_{p\in S}$ or $ \oy=(\oy_\infty, \oy_{p_1}, \ldots, \oy_{p_s})$, where $\oy_p\in \QQ_p^d$, the norm of the $S$-arithmetic space is defined by
\[
\|\oy\|_S=\max\{\|\oy_p\|_p: p\in S\}.
\]

Let us assign the measure $\vol$ on $\QQ_S^d$ as the product measure $\prod_{p\in S} \vol_p$, where $\vol_\infty$ is the usual Lebesgue measure and $\vol_p$ is the Haar measure for which $\vol_p(\ZZ_p^d)=1$.
The volume $\vol=\prod_{p\in S}\vol_p$ also stands for the Haar measure on $\Mat_{m,n}(\QQ_S)=\prod_{p\in S}\Mat_{m,n}(\QQ_p)$ by considering $\Mat_{m,n}(\QQ_p)\simeq \QQ_p^{mn}$ for each $p\in S$.
We simply denote $d\vol(\oy)$ by $d\oy$ in integral formulas.

Consider the diagonal embedding $\Delta:\QQ\rightarrow \QQ_S$ given as $\Delta(z)=(z,z,...,z)$. \emph{The set of rationals} in $\QQ_S$ is the image $\Delta(\QQ)$ of $\QQ$ under this diagonal embedding.
It is well-known that if we let $$\ZZ_S=\{z\in \QQ : |z|_\nu \le 1\;\text{for}\;\nu\notin S\}=\ZZ[1/p_1\cdots p_s],$$ the image $\Delta(\ZZ_S)$ is \emph{the ring of $S$-integers} in $\QQ_S$. For simplicity, we will use the notation $\ZZ_S$ instead of $\Delta(\ZZ_S)$. 

Notice that any ideal of $\ZZ_S$ is of the form $N\ZZ_S$ for some $N\in \NN_S$, where
\[\NN_S=\{N'\in \NN: \gcd(N', p_1\cdots p_s)=1\}.\]
Hence one can define \emph{a congruence condition} on $\ZZ_S$ as follows: 
\[z_1=z_2 \mod N
\Leftrightarrow
z_1-z_2\in N\ZZ_S.
\]

Let $\T=(T_p)_{p\in S}$ be an element of $\RR_{\ge 0}\times \prod_{p\in S_f} p^\ZZ$, where we define $p^\ZZ$ by $\{p^{k}: k\in \ZZ\}$. We say that $\T_1=(T^{(1)}_p)_{p\in S}\succeq \T_2=(T^{(2)}_p)_{p\in S}$ if $T^{(1)}_p \ge T^{(2)}_p$ (as positive real numbers) for all $p\in S$, and denote $\T\rightarrow \infty$ when $T_p\rightarrow \infty$ for all $p\in S$. 

For an element $\oy=(\oy_p)_{p\in S}\in \QQ_S^d$, we will denote $\|\oy_p\|_p$ by $\|\oy\|_p$ since it will not cause any confusion. Since we identify $\Delta(\ZZ_S^d)$ with $\ZZ_S^d$, we also use the notation $\ov=(\ov)_{p\in S}$ for an element $\ov\in\ZZ_S^d$. 

Let us settle a few more notations. For two sets $A$ and $B$, we will denote the set difference $A\cap B^c$ by $A-B$. If we say that $f\ll_n g$ for functions $f,\;g$ and a quantity $n$, it implies that there is a positive constant $C_n>0$ such that $f\le C_n g$, where $C_n$ depends only on $n$. Denote $f\asymp_n g$ when $f\ll_n g$ and $g \ll_n f$, etc.

\subsection*{$S$-arithmetic Analogs and Main Results}
Based on the proof of Dirichlet's theorem, 
since 
\[\begin{gathered}
\#\left\{\oq\in \ZZ_S^n : \|\oq\|_p\le T_p,\;\forall p\in S \right\}
=\Big(2\prod_{p\in S} T_p + 1\Big)^n;\\
\#\;\text{translates of}\; \Big( [0, T_\infty^{-\frac n m})\times 
\hspace{-0.07in}\prod_{p\in S_f} p^{z_p}\ZZ_p\Big)^m\text{in}\; \Big([0,1)\times \hspace{-0.07in}\prod_{p\in S_f} \ZZ_p\Big)^m \asymp_{S,m,n} \prod_{p\in S} T_p^n,
\end{gathered}\]
where $z_p\in \ZZ$ is the largest integer such that $p^{z_p}\le T_p^{n/m}$,
one can show that there is $C_p\ge 1$ for each $p\in S_f$ and $C_\infty=1$ such that the following holds: 
for any $A=(A_p)_{p\in S}\in \Mat_{m,n}(\QQ_S)$ and any $\T=(T_p)_{p\in S}$, $T_p\ge 1$, there is a nontrivial $S$-integral solution $(\op, \oq)\in \ZZ_S^m\times \ZZ_S^n$ satisfying the inequalities
\begin{equation}\label{New Dirichlet thm}
\|\oq\|_p \le T_p
\quad\text{and}\quad
\|A_p\oq+\op\|_p^m \le C_p T_p^{-n},\; \forall p\in S.
\end{equation}
One can therefore deduce that
there are infinitely many $(\op, \oq)\in \ZZ_S^m\times\ZZ_S^n$ for which
\begin{equation}\label{New Dirichlet cor}
\|A_p\oq+\op\|_p^m \le C_p \min\{1,\|\oq\|_p^{-n}\},\; \forall p\in S.
\end{equation}

In \cite{KT2007}, a different type of an $S$-arithmetic  Dirichlet theorem was previously introduced using the $S$-norm: 
there is $C>0$ such that for any $A\in \Mat_{m,n}(\QQ_S)$ and any $T>1$, one has a nontrivial $S$-integral solution $(\op,\oq)\in \ZZ_S^m\times \ZZ_S^n$ satisfying the inequalities
\begin{equation}\label{Old Dirichlet thm}
\|\oq\|_S\le T
\quad\text{and}\quad
\|A\oq + \op\|_S^m \le C T^{-n}.
\end{equation} 
As a corollary, one has that 
there are infinitely many $(\op, \oq)\in \ZZ_S^n\times\ZZ_S^m$ for which
\begin{equation}\label{Old Dirichlet cor}
\|A\oq+\op\|_S^m \le C\|\oq\|_S^{-n}.
\end{equation}

Let us remark that one hand, the $S$-arithmetic Dirichlet theorem given by \Cref{New Dirichlet thm} implies the one given by \Cref{Old Dirichlet thm}, and on the other hand, the $S$-arithmetic Dirichlet corollary given by \Cref{Old Dirichlet cor} implies the one given by \Cref{New Dirichlet cor}.

\begin{notation} We call $\psi=(\psi_p)_{p\in S}$ \emph{a collection of approximation functions} if for each $p\in S$, $\psi_p:\RR_{>0} \rightarrow \RR_{> 0}$ is a non-increasing function such that $\psi_p((0,1])\equiv 1$. When $p\in S_f$, let us further assume that for each $k\in \ZZ$, $\psi_p$ is constant on $\{p^{k'}: k'=kn, kn+1, kn+2, \ldots, kn+(n-1)\}$ and $\psi_p(p^{\ZZ})\in p^{m\ZZ}$, where $p^{\ZZ}:=\{p^k: k\in \ZZ\}$ and $p^{m\ZZ}:=\{p^k: k\in m\ZZ\}$.
\end{notation}

Here, the assumptions above are mild assumptions for notational simplicity. The theorems below would hold for $\psi=(\psi_p)_{p\in S}$, where each $\psi_p$ is a bounded non-increasing function with minor modifications.

Our first theorem shows the classical Khintchine-Groshev theorem with the new $S$-arithmetic setting based on \Cref{New Dirichlet cor}. The analogs related to \Cref{Old Dirichlet cor} can be deduced from \cite{KT2007, MG2009, DG2022}, which answer to the more delicate question suggested by Baker and Sprind$\breve{\text{z}}$uk, which is related to Diophantine approximation on manifolds. The original work for the real case was accomplished in \cite{KM1998}. See also \cite{GR2015} and \cite{DKL2005} for a classical and a quantitative Khintchine-Groshev theorem for function fields, respectively, and \cite{AGP2012} for a positive characteristic version.

\vspace{0.1in}
Our first result is the analog of the Khintchine-Groshev theorem in this setting.

\begin{theorem}\label{New Khintchine-Groshev thm} 
Assume $d=m+n\ge 3$.
Let $\psi=(\psi_p)_{p\in S}$ be a collection of approximating functions. Fix $N\in \NN_S$ and a pair $(\ov_m, \ov_n)\in \ZZ_S^m\times \ZZ_S^n$.

A system of inequalities
\[
\|A_p\oq +\op\|_p^m \le \psi(\|\oq\|_p^n),\quad \forall p\in S
\]
has infinitely many $S$-integer solutions 
\[(\op, \oq) \in \left\{(\ow_m, \ow_n)\in \ZZ_S^m\times \ZZ_S^n : (\ow_m, \ow_n)\equiv (\ov_m, \ov_n)\mod N\right\}
\]
\begin{enumerate}
\item for almost no $A\in \Mat_{m,n}(\QQ_S)$
$\Leftrightarrow$
$\int_{\QQ_S^n} \prod_{p\in S} \psi_p(\|\oy\|^n_p) d\oy <\infty$;
\item for almost all $A\in \Mat_{m,n}(\QQ_S)$
$\Leftrightarrow$
$\int_{\QQ_S^n} \prod_{p\in S} \psi_p(\|\oy\|^n_p) d\oy =\infty$.
\end{enumerate}
\end{theorem}

Here since $\int_{\QQ_S^n} \prod_{p\in S} \psi_p d\oy =\prod_{p\in S} \int_{\QQ_p^n} \psi_p d\oy_p$ (and each $\psi_p$ is a non-negative function with $\int_{\QQ_p^n} \psi_p d\oy_p>0$), the integral diverges if and only if there is $p\in S$ such that $\int_{\QQ_p^n} \psi_p d\oy_p =\infty$.
Equivalently,
\[
\lim_{\T\rightarrow\infty}
\int_{\{\oy\in \QQ_S^n: \|\oy\|^n_p\le T_p,\;\forall p\in S\}}
\prod_{p\in S} \psi_p (\|\oy\|_p^n) d\oy =\infty.
\]

The convergent part of the above theorem follows using the classical Borel-Cantelli lemma so that it holds even for the case when $m=n=1$.
The divergent case is a direct consequence of the theorem below.

\begin{theorem}\label{main thm 1}
Assume $m+n\ge 3$.
Let $\psi=(\psi_p)_{p\in S}$ be a collection of approximating functions.
Let $N\in \NN_S$ and fix a pair $(\ov_m, \ov_n)\in \ZZ_S^m\times \ZZ_S^n$.

For $A=(A_p)_{p\in S}\in \Mat_{m,n}(\QQ_S)$ and $\T=(T_p)_{p\in S}$, consider the functions
\[\begin{gathered}
\NT_{\psi,A}(\T):=\#
\left\{(\op,\oq)\in \ZZ_S^m\times \ZZ_S^n:\begin{array}{c}
\circ\:(\op,\oq)\equiv(\ov_m,\ov_n)\mod N,\;\text{and}\\[0.05in]
\hspace{-1.125in}\circ\:\text{for each }p\in S,\\
 \hspace{0.3in}\| A_p\oq+\op\|^m_p \le \psi_p(\|\oq\|_p^n);\\
 \hspace{0.23in}\|\oq\|_p^n \le T_p\end{array}\right\};\\
\VT_{\psi}(\T):=2^m\int_{\{\oy\in \QQ_S^n: \|\oy\|^n_p\le T_p\}} \prod_{p\in S} \psi_p(\|\oy\|^n_p) d\oy. 
\end{gathered}\]

Consider a sequence $(\T_r)_r$, where the index set is a subset of $\RR_{\ge0}$, such that
\[
\T_{r_2} \succeq \T_{r_1}\;\text{if}\;r_2\ge r_1
\quad\text{and}\quad
\lim_{r\rightarrow\infty}\VT_{\psi}(\T_r)=\infty.
\] 

For almost all $A\in \Mat_{m,n}(\QQ_S)$, it follows that
\[
\lim_{r\rightarrow\infty} \frac {\NT_{\psi,A}(\T_r)}{\VT_{\psi}(\T_r)/N^d}=1.
\]
\end{theorem}

Note that if $\int_{\QQ_S^n} \prod_{p\in S} \psi_p(\|\oy\|_p^n)d\oy<\infty$, the condition of \Cref{main thm 1} and also that of \Cref{main thm 2} below are not satisfied.

The reason that we can show only for the limit along increasing sequences in the above theorem is that our method needs to control the number of $\T$ for which $\VT_\psi(\T)\asymp k^\alpha$ for some $\alpha>0$, as $k\in \NN$ goes to infinity (see \Cref{Analogue of Schmidt 2}).
If approximating functions in $\psi$ decrease slowly enough, one can show that the counting function $\NT_{\psi,A}(\T)$ converges asymptotically to the volume form $\VT_{\psi}(\T)$ as $\T$ goes to infinity, for almost all $A\in \Mat_{m,n}(\QQ_S)$.

For the convenience, let us extend the variable $\T$ to have infinities, so that for $\T\in (\RR_{\ge 1}\cup \{\infty\})\times \prod_{p\in S_f} p^{\NN\cup\{\infty\}}$, one can define
\[
\VT_\psi(\T)
=2^m\int_{\{\oy\in \QQ_S^n: \|\oy\|_p^n\le T_p,\;\forall p\in S\;\text{with}\;T_p\neq \infty\}} \prod_{p\in S} \psi_p(\|\oy\|_p^n) d\oy,
\]
which possibly admidts an infinite value.

\begin{theorem}\label{main thm 2}
Under the same assumptions with \Cref{main thm 1}, suppose that there are $\delta_1,\;\alpha>0$ for which $\delta_1+1<\alpha<\delta_1+3$ and $C_\psi>0$ (also depending on $\delta_1$ and $\alpha$) such that
\[
\#\left\{\T\in (\RR_{\ge 1}\cup\{\infty\})\times \prod_{p\in S_f} p^{\NN\cup\{0,\infty\}}:
\begin{array}{c} 
\frac{\VT(\T)}{N^d}\in [k^\alpha, (k+1)^\alpha],\;\text{and}\\[0.05in]
\T\text{ is $(N,k,\alpha)$-maximal or}\\
\text{$(N,k,\alpha)$-minimal}
\end{array} \right\}< C_\psi k^{\delta_1}
\]
for any $k\in \NN$, where we say that $\T$ is \emph{$(N,k,\alpha)$-maximal ($(N,k,\alpha)$-minimal, respectively)} if
\[
\nexists \T' \text{ s.t. } \VT_\psi(\T')/N^d\in [k^\alpha, (k+1)^\alpha]
\;\text{and}\; \T' \succ \T\;(\T' \prec \T,\text{ respectively}).  
\]

For almost all $A\in \Mat_{m,n}(\QQ_S)$, it follows that
\[
\lim_{\T\rightarrow\infty} \frac {\NT_{\psi,A}(\T)}{\VT_{\psi}(\T)/N^d}=1.
\]
\end{theorem}

For instance, the collection $\psi=(\psi_p)_{p\in S}$ of approximating functions given by
\[
\psi_p(p^{nk_p})\asymp p^{-\ell_p k_p},\;\forall k_p\in \NN
\]
for some $\ell_p<n$ ($p\in S_f$) and any function $\psi_\infty$ satisfies the condition in \Cref{main thm 2}, but the case when $\ell_p=n$ could be not our example because of the absence of such a pair $(\delta_1, \alpha)$ (see volume formulas in \Cref{Volume Formula} and also explanation in \Cref{Dirichlet case}).

%
%

\vspace{0.1in}
The paper is organized as follows. 
In \Cref{Volume Formula}, we compute the volume formula of the region bounded by the inequalities $\|\ox\|^m_p\le \psi_p(\|\oy\|_p^n)$ and $\|\oy\|^n_p\le T_p$ for each $p\in S$, which will be used for obtaining \Cref{main thm 1} and \Cref{main thm 2}.
In \Cref{Sec: Convergence Case}, we will show the convergent case of \Cref{New Khintchine-Groshev thm}. From the classical Borel-Cantelli lemma, one can show that the \emph{almost no} statement holds when $\sum_{\oq\in \ZZ_S^n} \prod_{p\in S} \psi_p(\|\oq\|_p^n)<\infty$, which is the same condition as in \cite{MG2009} (although they use a single function $\psi$ instead of the collection of functions $(\psi_p)_{p\in S}$). And then using the formula in \Cref{Volume Formula}, we will see that this condition is equivalent to say that $\int_{\QQ_S^n} \prod_{p\in S}\psi(\|\oy\|_p^n) d\oy<\infty$.
In \Cref{Sec: Divergence Case: The Quantitative Result}, we first show two $S$-arithmetic analogs of Schmidt's theorem \cite[Theorem 1]{Schmidt1960b} for the special case, which fit with our situations. As a consequence, we prove \Cref{main thm 1} and \Cref{main thm 2}.
In the last section, let us introduce a few possible questions to be thought of and explain technical reasons why they are not dealt with in this article.

\subsection*{Acknowledgements}
I would like to thank Anish Ghosh for valuable advice and discussion. 
The article was partially done during the author was at Tata Institute of Fundamental Research. 
I am also very grateful to an anonymous referee for many insightful comments and helpful suggestions.
This project is supported by a KIAS Individual Grant MG088401 at Korea Institute for Advanced Study. 

\section{Volume Formula}\label{Volume Formula}
In this short subsection, let us compute the volume of the following set
\begin{equation}\label{region}
E_{\psi}(\T)=\left\{(\ox,\oy)\in \QQ_S^m\times \QQ_S^n:
\begin{array}{c}
\hspace{-0.4in}\text{For each }p\in S,\\
 \hspace{0.3in}\| \ox\|^m_p \le \psi_p(\|\oy\|_p^n)\;\text{and}\\
 \hspace{0.3in}\|\oy\|^n_p \le T_p\end{array}\right\},
\end{equation}
which will be equal to $\VT_{\psi}(\T)$ defined as in \Cref{main thm 1}.

We have that $E_{\psi}(\T)=\prod_{p\in S} E_{\psi_p}(T_p)$, where
\[
E_{\psi_p}(T_p)=\left\{(\ox,\oy)\in \QQ_p^m\times \QQ_p^n :
\|\ox\|_p^m \le \psi_p(\|\oy\|_p^n)\;\text{and}\; \|\oy\|_p^n\le T_p \right\}.
\]

Hence 
\[\begin{split}
\vol\left(E_\psi(\T)\right)&=\prod_{p\in S}\int_{E_{\psi_p}(T_p)} 1 d\ox d\oy\\
&=\prod_{p\in S}\int_{\{\oy\in \QQ_p^n : \|\oy\|_p^n \le T_p \}}
\int_{\{\ox\in \QQ_p^m : \|\ox\|_p^m \le \psi_p(\|\oy\|_p^n)^{1/m}\}} 1 d\ox d\oy\\
&=\int_{\{\oy\in \RR^n: \|\oy\|_\infty^n\le T_\infty\}} 2^m\psi_\infty(\|\oy\|_\infty^n) d\oy\\
&\hspace{0.4in}\times\prod_{p\in S_f}
\int_{\{\oy\in \QQ_p^n: \|\oy\|_p^n \le T_\infty\}} \psi_p(\|\oy\|_p^n) d\oy\\
&=2^m \int_{\{\oy\in \QQ_S^n: \|\oy\|_p^n\le T_p,\; \forall p\in S\}} \prod_{p\in S} \psi_p(\|\oy\|_p^n) d\oy=\VT_{\psi}(\T).
\end{split}\]

For the next section, let us note that the inner integrals above can be expressed as
\begin{equation}\label{another formula}
\begin{gathered}
\int_{\{\oy\in \RR^n : \|\oy\|_\infty^n \le T_\infty\}} \psi_\infty(\|\oy\|_\infty^n) d\oy
=2^n\int_0^{T_\infty} \psi_\infty(r) dr;\\
\int_{\{\oy\in \QQ_p^n : \|\oy\|_p^n \le T_p\}} \psi_p(\|\oy\|_p^n) d\oy=\sum_{k_p=-\infty}^{t_p} p^{k_pn}\left(1 -\frac 1 {p^n}\right) \psi_p(p^{k_pn}),
\end{gathered}
\end{equation}
where $T_p=p^{t_pn}$ for $p\in S_f$.

\section{Convergent Case}\label{Sec: Convergence Case}

For the convergent part, we may assume that $N=1$ since for any $A=(A_p)_{p\in S}\in \Mat_{m,n}(\QQ_S)$, there is an inclusion between the solution sets
\[\begin{split}
&\left\{(\op,\oq)\in \ZZ_S^m\times \ZZ_S^n: \begin{array}{c}
\|A_p\oq+\op\|^m_p\le \psi_p(\|\oq\|_p^n),\;\forall p\in S\\
(\op,\oq)=(\ov_m,\ov_n)\mod N \end{array}\right\}\\
&\hspace{0.4in}\subseteq
\left\{(\op,\oq)\in \ZZ_S^m\times \ZZ_S^n: \|A_p\oq+\op\|^m_p\le \psi_p(\|\oq\|_p^n),\;\forall p\in S\right\}
\end{split}\]
for any $N\in \NN_S$ and $(\ov_m, \ov_n)\in \ZZ_S^m\times \ZZ_S^n$.

\begin{proof}[Proof of \Cref{New Khintchine-Groshev thm} (1)]
Using the fact that $\QQ_S/\ZZ_S\simeq [0,1)\times \prod_{p\in S_f} \ZZ_p$, it is enough to show that the following set
\[
\mathcal A_{m,n}(\psi)
=\left\{ A\in \Mat_{m,n}\Big([0,1)\times \prod_{p\in S_f} \ZZ_p\Big): \begin{array}{c}
\|A_p\oq + \op\|_p^m \le \psi_p(\|\oq\|_p^n),\;\forall p\in S\\[0.05in]
\text{for infinitely many}\\[0.05in]
 (\op,\oq)\in \ZZ_S^m\times \ZZ_S^n \end{array}\right\}
\]
has measure zero when $\int_{\QQ_S^n} \prod_{p\in S} \psi_p(\|\oy\|_p^n) d\oy<\infty$.

For each $\oq\in \ZZ_S^n$, define $\mathcal A_{\oq}=\mathcal A_{\oq}(\psi)$ by
\[
\mathcal A_{\oq}=\left\{A\in \Mat_{m,n}\Big([0,1)\times \prod_{p\in S_f} \ZZ_p\Big) : 
\begin{array}{c}
\|A_p\oq + \op\|_p^m \le \psi_p(\|\oq\|_p^n),\; \forall p\in S \\[0.05in]
\text{for some}\; \op\in \ZZ_S^m \end{array}\right\}.
\]
It is easy to verify that $\mathcal A_{m,n}(\psi)\subseteq \limsup_{\oq} \mathcal A_{\oq}$. Hence to use the Borel-Cantelli lemma, we need to show that
\[
\sum_{\oq\in \ZZ_S^n-\{\origin\}} \vol(\mathcal A_{\oq}) < \infty.
\]

We first want to count the number of $\oq=(\oq)_{p\in S}\in \ZZ_S^n$ for which $\|\oq\|_p=T_p$ for any $p\in S$, when $\T=(T_p)_{p\in S}$ is given.
Denote
\begin{equation}\label{Labeling}
T_\infty=\ell_{\infty}p_1^{\ell_1}\cdots p_s^{\ell_s}
\quad\text{and}\quad 
T_{p_i}={p_i}^{k_i},
\end{equation}
where $\ell_\infty\in \NN_S$ and $\ell_i, \;k_i\in \ZZ$ for $1\le i \le s$.
Any possible $\oq=(q_1, \ldots, q_n)$ is 

\noindent 1) in $\frac 1 {p_1^{k_1}\cdots p_s^{k_s}} \ZZ^n \cap B_{\ell_\infty p_1^{\ell_1} \cdots p_s^{\ell_s}} (\origin)$, where 
$$B_{\ell_\infty p_1^{\ell_1} \cdots p_s^{\ell_s}} (\origin)=[-\ell_\infty p_1^{\ell_1} \cdots p_s^{\ell_s},\ell_\infty p_1^{\ell_1} \cdots p_s^{\ell_s}]^n;$$

\noindent 2) for some $1\le j\le n$, $q_j=\pm \ell_\infty p_1^{\ell_1} \cdots p_s^{\ell_s}$.

Let us assume that $|q_1|_\infty=T_\infty=\ell_\infty p_1^{\ell_1} \cdots p_s^{\ell_s}$.
From the condition that
$|q_1|_{p_i}=p_i^{-\ell_i} \le \|\oq\|_{p_i}=p_i^{k_i}$,
we further obtain the condition for $\T$:
\begin{equation}\label{cond for T}
\ell_i \ge -k_i\quad\text{for any}\; 1\le i \le s.
\end{equation}

When $-\ell_j=k_j$ for all $1\le j\le s$, i.e., if $|q_1|_p=T_p$ for each $p\in S_f$, then the number of $\oq\in \ZZ_S^n$ for which $(\|\oq\|_p)_{p\in S}=\T$ is 
\begin{equation}\label{no of oq}
2n(2\prod_{p\in S}T_p+1)^{n-1}.
\end{equation}
This quantity turns out to be the case of having the largest upper bound, 
since if $|q_1|_{p_j}\neq T_{p_j}$ for some $1\le j\le s$, then $|q_i|_{p_j}=p_j^{k_j}$ for some $2\le i\le n$, hence the additional condition that $q_i=m_i p_j^{k_j}$ with $(m_i, p_j)=1$ reduces the number of such $\oq$'s.

\vspace{0.1in}
Next, let us compute the upper bound of $\vol(\mathcal A_\oq)$ for a given $\oq\in \ZZ_S^n$ with $(\|\oq\|_p)_{p\in S}=\T$, which is equal to $\vol(\mathcal X_\oq)^m$, where
\[\begin{split}
\mathcal X_\oq
&=\left\{\begin{array}{lr}
X=(x_1, \ldots, x_n)\in \Big([0,1)\times \prod_{p\in S_f} \ZZ_p\Big)^n : \\[0.05in]
\;\exists b\in \ZZ_S\;\text{s.t. }|(x_1q_1+ \cdots + x_nq_n)+b|_p \le \psi_p(\|\oq\|_p^n)^{1/m},\;\forall p\in S\;\end{array} \right\}\\
&=\bigcup_{b\in \ZZ_S} \mathcal X_{\oq, b},
\end{split}\]
and $\mathcal X_{\oq, b}$ is the set of $X=(x_1,\ldots, x_n)\in \mathcal X_\oq$ for which $|(x_1q_1+ \cdots + x_nq_n)+b|_p \le \psi_p(\|\oq\|_p^n)^{1/m}$ for $p\in S$.

Note that since $\|(x_1, \ldots, x_n)\|_p\le 1$ and $0\le \psi_p\le 1$ for each $p\in S$, the element $b\in \ZZ_S$ such that $\mathcal X_{\oq,b}\neq \emptyset$ satisfies that
\[\begin{gathered}
|b|_\infty \le n\|\oq\|_\infty+1
\le 2n\max\{\|\oq\|_\infty, 1\};\\[0.05in]
|b|_p\le \max\{\|\oq\|_p, \psi_p(\|\oq\|_p^n)^{1/m}\}
\le \max\{\|\oq\|_p, 1\}
\end{gathered}\]
for $p\in S_f$. Hence, we obtain the upper bound
\[
\# \{ b\in \ZZ_S : \mathcal X_{\oq, b}\neq \emptyset\}
\le 2n\prod_{p\in S} \max \{\|\oq\|_p, 1\}.
\]

The volume $\vol(\mathcal X_{\oq, b})$ is the product of
\[\begin{gathered}
\vol_\infty\Big(\Big\{(x_1,\ldots, x_n)\in [0,1)^n:
\Big|\sum_{i=1}^n x_iq_i+b\Big|_\infty \le \psi_\infty(\|\oq\|_p^n)^{1/m}\Big\}\Big) \quad\text{and}\\
\vol_p\Big(\Big\{(x_1, \ldots, x_n)\in \ZZ_p^n:
\Big|\sum_{i=1}^n x_iq_i+b\Big|_p \le \psi_p(\|\oq\|_p^n)^{1/m}\Big\}\Big)
\end{gathered}\]
for all $p\in S_f$. 
It is well-known that an upper bound of the volume above for the infinite place can be taken to be $2\psi_\infty(\|\oq\|_\infty^n)^{1/m}/\|\oq\|_\infty$ if $\|\oq\|_\infty >1$.
When $\|\oq\|_\infty \le 1$, we are not interested in the volume of the the given set and will take an upper bound as 1 (which is the volume of $[0,1)^n$).

For the case when $p<\infty$, let $|q_{i_0}|_p=\|\oq\|_p$ for some $1\le i_0\le n$. 
Then we have that
\[
x_{i_0} \in \ZZ_p \cap \Big(\psi_p(\|\oq\|_p^n)^{-1/m}q_{i_0}^{-1}\ZZ_p  -b- q_{i_0}^{-1}\sum_{1\le i\neq i_0\le n}x_iq_i\Big)
\]
so that if $\|\oq\|_p>1$, the volume is bounded above by
\[\begin{split}
&\int_{\ZZ_p^{n-1}} \int_{\psi_p(\|\oq\|_p^n)^{-1/m}q_{i_0}^{-1}\ZZ_p  -b- q_{i_0}^{-1}\sum\limits_{1\le i\neq i_0\le n}x_iq_i} 1 dx_{i_0} dx_1\cdots dx_{i_0-1}dx_{i_0+1}\cdots dx_{n}\\
&\hspace{2.in}=\psi_p(\|\oq\|_p^n)^{1/m}| q^{-1}_{i_0}|_p
=\psi_p(\|\oq\|_p^n)^{1/m}/\|\oq\|_p,
\end{split}\]
and if $\|\oq\|_p\le 1$, as in the case of the infinite place, we will give the upper bound of the volume as 1.

Therefore, we have the following upper bound of the volume of $\mathcal X_{\oq}$:
\begin{equation}\label{vol of Xq}\begin{split}
\vol(\mathcal X_\oq)
&\le 2n \prod_{p\in S}\max\{\|\oq\|_p,1\} 
\times
\prod_{p\in S} \frac{\psi_p(\|\oq\|_p^n)^{1/m}}{\max\{\|\oq\|_p, 1\}}\\
&=2n \prod_{p\in S} \psi_p(\|\oq\|_p^n)^{1/m}.
\end{split}\end{equation}

By Equations \eqref{cond for T}, \eqref{no of oq}, and \eqref{vol of Xq}, using the notation \Cref{Labeling} for $(T_p)_{p\in S}$, we obtain that
\[\begin{split}
\sum_{\oq\in \ZZ_S^n-\{\origin\}} \vol(\mathcal A_\oq)
&\le\sum_{k_1\in \ZZ}\cdots \sum_{k_s\in \ZZ}
\hspace{-0.1in}\sum_{\scriptsize \begin{array}{c}
\ell_j\ge -k_j\\
(1\le j\le s);\\
\ell_\infty\in \NN_S\end{array}}\hspace{-0.15in}
(2n)^{m+1}\Big(2\prod_{p\in S}T_p + 1\Big)^{n-1}\prod_{p\in S} \psi_p(T_p^n).
\end{split}\]
Now, since $|T_\infty|_p\le T_p$ for each $p\in S_f$, $\prod_{p\in S}T_p \ge \prod_{p\in S} |T_\infty|_p=\ell_\infty \ge 1$ (by considering $T_\infty$, $T_{p_1}, \ldots, T_{p_s}$ as rational numbers), so that it suffices to show that
\begin{equation}\label{vol form 1}
\sum_{k_1\in \ZZ}\cdots \sum_{k_s\in \ZZ}
\hspace{-0.1in}\sum_{\scriptsize \begin{array}{c}
\ell_j\ge -k_j\\
(1\le j\le s);\\
\ell_\infty\in \NN_S\end{array}}
\Big(\prod_{p\in S} T_p\Big)^{n-1} \prod_{p\in S}\psi_p(T_p^n)
<\infty.
\end{equation}

For each $k_1,\ldots, k_s\in \ZZ$, let us take $K=\prod_{p\in S_f} T_p=\prod_{1\le i \le s} p_i^{k_i}$ for a while. The inner sum over $T_\infty=\ell_\infty p_1^{\ell_1}\cdots p_s^{\ell_s}$ in \Cref{vol form 1} is then
\[\begin{split}
\sum_{\scriptsize \begin{array}{c}
\ell_j\ge -k_j\\
(1\le j\le s);\\
\ell_\infty\in \NN_S\end{array}} T_\infty^{n-1} \psi_{\infty}(T_\infty^n)
&=\sum_{T_\infty\in \frac 1 {K}\NN}
T_\infty^{n-1}\psi_\infty(T_\infty^n)
=\sum_{m\in \NN} \Big(\frac m {K}\Big)^{n-1}\psi_\infty\left(\Big(\frac m {K}\Big)^n\right)\\
&\ll_{n} \sum_{m\in \NN} \frac 1 {K^{n-1}} \psi_{\infty} \left(\frac m {K^n}\right)
={K} \sum_{m\in \NN} \frac 1 {K^n} \psi_\infty \left(\frac m {K^n}\right)\\
&\le (T_1\cdots T_s) \int_0^\infty \psi_\infty(T_\infty) dT_\infty,
\end{split}\] 
where in the last inequality, we use that $\psi_\infty$ is non-increasing.

Hence the summation in \Cref{vol form 1} is bounded by
\[\begin{split}
&\ll_{n} \prod_{p\in S_f} \Big(\sum_{k_p\in \ZZ} p^{nk_p} \psi_p(T_p^n)\Big)
\times
\int_0^\infty \psi_\infty(T_\infty) dT_\infty \\
&\ll_{n,S}\int_{\QQ_S^n} \prod_{p\in S} \psi_p(\|\oy\|_p^n) d\oy
 <\infty
\end{split}\]
by using \Cref{another formula} and the fact that the integral of $\prod_{p\in S} \psi_p$ is finite. Therefore we conclude that $\vol(\mathcal A_{m,n}(\psi))=0$ by the Borel-Cantelli lemma.
\end{proof}

\section{Divergence Case: The Quantitative Result}\label{Sec: Divergence Case: The Quantitative Result}

We want an asymptotic formula for 
\[\begin{split}
\NT_{\psi,A}(\T)
&=\#
\left\{(\op,\oq)\in \ZZ_S^m\times \ZZ_S^n:\begin{array}{c}
\circ\:(\op,\oq)\equiv(\ov_m,\ov_n)\mod N,\;\text{and}\\[0.05in]
\hspace{-1.125in}\circ\:\text{for each }p\in S,\\
 \hspace{0.3in}\| A_p\oq+\op\|^m_p \le \psi_p(\|\oq\|_p^n);\\
 \hspace{0.23in}\|\oq\|_p^n \le T_p\end{array}\right\}\\
&=\#\su_A\left(N\ZZ_S^d+(\ov_m,\ov_n)\right) \cap E_\psi(\T),
\end{split}\]
where $\su_A:=\left(\begin{array}{cc}
I_m & A \\
0 & I_n \end{array}\right)$, $\psi=(\psi_p)_{p\in S}$ is a collection of approximating functions, $N\in \NN_S$, $(\ov_m, \ov_n)\in \ZZ_S^m\times \ZZ_S^n$, and $E_{\psi}(\T)$ is the set given in \Cref{region}, when the integral of $\prod_{p\in S}\psi_p$ diverges. Set $\T'=(T'_p)_{p\in S}$ and $\psi'=(\psi'_p)_{p\in S}$ such that
\begin{equation}\label{psi'}
\left\{\begin{array}{c}
T'_\infty= T_\infty/N^n; \\
T'_p=T_p\;(p\in S_f);
\end{array}\right.
\quad
\left\{\begin{array}{c}
\psi'_\infty(\|\oy\|_\infty^n):= \psi_\infty(\|N\oy\|_\infty^n)/N^m;\\
\psi'_p=\psi_p\;(p\in S_f)
\end{array}\right.
\end{equation}
so that it holds that
\[
\#\su_A\left(N\ZZ_S^d+(\ov_m,\ov_n)\right) \cap E_\psi(\T)
=\# \su_A\left(\ZZ_S^d+ \frac 1 N (\ov_m, \ov_n)\right) \cap E_{\psi'}(\T').
\]

Put $\ov_d=(\ov_m, \ov_n)$ and define
\[
Y_{\ov_d/N}=\left\{\sg\left(\ZZ_S^d + \frac {\ov_d} {N}\right) : \sg \in \SL_d(\QQ_S) \right\},
\]
where $\SL_d(\QQ_S)=\prod_{p\in S}\SL_d(\QQ_p)$. Let $\mu_S$ be the Haar measure on $\SL_d(\QQ_S)$ for which $\mu_S(Y_{\ov_d/N})=1$, where we use the same notation $\mu_S$ for the push-forward measure of $\mu_S$ under the projection $\SL_d(\QQ_S)\rightarrow Y_{\ov_d/N}$.

We first examine the following asymptotic formulas for random affine lattices in $Y_{\ov_d/N}$, which are generalizations of the Schmidt's theorem \cite{Schmidt1960b}. 

\begin{theorem}\label{Analogue of Schmidt 1}
Let $\psi'=(\psi'_p)_{p\in S}$ be as in \Cref{psi'}
and take any $\delta\in \left(\frac 2 3,1\right)$. 
Consider a sequence $(\T_r)_{r}$, where the index set is a subset of $\RR_{\ge 0}$, such that 
\[
\T_{r_2}\succeq \T_{r_1}\;\text{if}\;r_2\ge r_1
\quad\text{and}\quad 
\lim_{r\rightarrow\infty} \vol(E_{\psi'}(\T_r))= \infty.
\]

Then the sequence $(\T'_r)_r$, which is defined from $(\T_r)_r$ as in \Cref{psi'}, also has the same properties above.

For a.e. $\sg\in \SL_d(\QQ_S)$, it holds that
\[
\#\sg\left(\ZZ_S^d+\frac {\ov_d} N\right) \cap E_{\psi'}(\T'_r) = \vol(E_{\psi'}(\T'_r)) + O(\vol(E_{\psi'}(\T'_r))^{\delta})
\]
for all sufficiently large $r$ (depending on $\sg$). Hence it follows that for a.e. $\sg\in \SL_d(\QQ_S)$, 
\[
\#\sg\left(N\ZZ_S^d+\ov_d\right)\cap E_{\psi}(\T_r)
=\frac 1 {N^d} \vol\left(E_\psi(\T_r)\right)+O_N(\vol\left(E_{\psi}(\T_r)\right)^\delta)
\]
for all sufficiently large $r$.
\end{theorem}

For the next theorem, let us allow the infinite value for $T_p$ ($p\in S$), and define
\[
E_{\psi'_p}(\infty):=\left\{(\ox,\oy)\in \QQ_p^m\times \QQ_p^n :
\|\ox\|_p^m \le \psi'_p(\|\oy\|_p^n) \right\}.
\]
This set will be considered only when $\vol_p(E_{\psi'_p}(\infty))<\infty$.

\begin{theorem}\label{Analogue of Schmidt 2}
Let $\psi'=(\psi'_p)_{p\in S}$ and $\T'=(T'_p)_{p\in S}$ be as in \Cref{psi'}.
Suppose that there are $(\delta_1, \delta, \alpha)\in \left(0,\infty\right)\times \left(\frac {\delta_1+2} {\delta_1+3}, 1\right)\times \left(\frac {1+\delta_1} {2\delta-1}, \frac 1 {1-\delta}\right)$ and $C_{\psi'}>0$ such that
\[\begin{split}
&\#\left\{\T'\in (\RR_{\ge \frac 1 {N^n}}\cup\{\infty\})\times\prod_{p\in S_f} p^{\NN\cup\{0,\infty\}}:
\begin{array}{c}
\vol(E_{\psi'}(\T'))\in [k^\alpha, (k+1)^\alpha],\;\text{and}\\
E_{\psi'}(\T')\;\text{is $(k,\alpha)$-maximal or}\\
 \text{$(k,\alpha)$-minimal}
\end{array} \right\}\\
&\hspace{0.2in}< C_{\psi'} k^{\delta_1}
\end{split}\]
for any $k\in \NN$, where $E_{\psi'}(\T')$ is \emph{$(k,\alpha)$-maximal ($(k,\alpha)$-minimal, respectively)} if 
\[\begin{split}
\nexists \T'' \text{ s.t. }\vol(E_{\psi'}(\T''))\in [k^\alpha, (k+1)^\alpha]& \text{ with }E_{\psi'}(\T')\subsetneq E_{\psi'}(\T'')\\
(&E_{\psi'}(\T')\supsetneq E_{\psi'}(\T''),\text{ respectively}).
\end{split}\]

For a.e. $\sg\in \SL_d(\QQ_S)$, we have that
\[
\#\sg\left(\ZZ_S^d+\frac {\ov_d} N\right) \cap E_{\psi'}(\T') = \vol(E_{\psi'}(\T')) + O(\vol(E_{\psi'}(\T'))^{\delta})
\]
for all sufficiently large $\T'$, hence, for all sufficiently large $\T$, it holds that
\[
\#\sg\left(N\ZZ_S^d+\ov_d\right)\cap E_{\psi}(\T)
=\frac 1 {N^d} \vol\left(E_\psi(\T)\right)+O_N(\vol\left(E_{\psi}(\T)\right)^\delta).
\]
\end{theorem}

%

Notice that for any positive $\delta_1, \alpha>0$ with $\delta_1+1<\alpha<\delta_1+3$, any $\delta\in \left(\max\left\{1-\frac 1 \alpha, \frac 1 2(\frac {1+\delta_1} \alpha +1), \frac {\delta_1+2} {\delta_1+3}\right\},1\right)$ satisfies that $(\delta_1, \delta, \alpha)\in  \left(0,\infty\right)\times \left(\frac {\delta_1+2} {\delta_1+3}, 1\right)\times \left(\frac {1+\delta_1} {2\delta-1}, \frac 1 {1-\delta}\right)$.

To prove the theorems above under our strategy, we need followings.

\begin{theorem}\label{GH22 Theorem 5.2}
Let $d\ge 3$ and let $E=\prod_{p\in S} E_p$ be the product of Borel sets $E_p\subseteq \QQ_p^d$ for $p\in S$ with $\vol(E)<\infty$.
There is a constant $C_d>0$, depending only on the dimension $d$, such that
\[
\mu_S\left(\left\{\Lambda\in Y_{\ov_d/N} : \left|\#(\Lambda \cap E) - \vol(E) \right|>M\right\}\right)< C_d \frac {\vol(E)} {M^2}
\]
for any positive $M>0$.
\end{theorem}
\begin{proof} For a bounded set $E=\prod_{p\in S_f} E_p$, the result follows from \cite[Theorem 5.2]{GH22}. 

If $E$ is unbounded, one can take a sequence $(E_k=\prod_{p\in S}E^{(k)}_p)_{k\in \NN}$ converging to $E$. Since
\[\begin{split}
&\left\{\Lambda\in Y_{\ov_d/N} : \left|\#(\Lambda \cap E) - \vol(E) \right|>M\right\}\\
&\hspace{1in}=\liminf_{\scriptsize \begin{array}{c}
k\rightarrow\infty\\
\exists E^{\pm}(k)
\end{array}}
\left\{\Lambda\in Y_{\ov_d/N} : \left|\#(\Lambda \cap E_k) - \vol(E_k) \right|>M\right\},
\end{split}\]
the theorem is deduced from Fatou's lemma. 
\end{proof}

It is unclear that for a given non-increasing function $\psi_0:\RR_{>0}\rightarrow \RR_{> 0}$, there is a constant $C>0$ such that
\[
\mu_S\left(\left\{\Lambda \in Y_{\ov_d/N} : |\#(\Lambda\cap E_{\psi_0}(T))-\vol(\mathsf E_{\psi_0}(T))|>M\right\}\right) < C \frac {\vol(\mathsf E_{\psi_0}(T))} {M^2}
\]
for all sufficiently large $T>1$, where 
\[
\mathsf E_{\psi_0}(T):=\left\{(\ox,\oy)\in \QQ_S^m\times \QQ_S^n : \|\ox\|_S^m \le \psi_0(\|\oy\|_S^n)\;\text{and}\;\|\oy\|_S^n \le T \right\}.
\]
Unlike the set $\mathsf E_{\psi_0}(T)$, the set $E_\psi(\T)$ is the product of Borel sets in $\QQ_p^d$ for $p\in S$ so that \Cref{GH22 Theorem 5.2} can be applicable. This is the reason of considering a new type of an $S$-arithmetic Khintchine-Groshev theorem. 

\vspace{0.1in}
For a discrete set $\Lambda\subseteq \QQ_S^d$ and a measurable set $E\subseteq \QQ_S^d$ with finite volume, define 
\[ D(\Lambda, E)=|\# (\Lambda \cap E) -\vol(E)|.\]

The following lemma is easy to obtain.

\begin{lemma}\label{discrepency}
Let $E_1 \subseteq E \subseteq E_2 \subseteq \QQ_S^d$ be measurable sets with finite volume. Then
\[
D(\Lambda, E)+\vol(E_2- E_1) \le \max\left\{D(\Lambda, E_1), D(\Lambda, E_2)\right\}
\]
for any discrete set $\Lambda\subseteq \QQ_S^d$.
\end{lemma}

\begin{proof}[Proof of \Cref{Analogue of Schmidt 1}]

For a given $\delta\in \left(\frac 2 3, 1\right)$, choose $\alpha\in \left(\frac 1 {2\delta-1}, \frac 1 {1-\delta}\right)$.
In our strategy, we only concern the set $\{\vol(E_{\psi'}(\T'_r))\}_r$. 

For each $k\in \NN$ with $\{\vol(E_{\psi'}(\T'_r))\}_r \cap [k^\alpha, (k+1)^\alpha]\neq \emptyset$, define the sets $E^+(k)$ and $E^-(k)$:
\[\begin{split}
E^+(k)&=\bigcup\left\{E_{\psi'}(\T'_r): \vol(E_{\psi'}(\T'_r))\in [k^\alpha, (k+1)^\alpha]\right\};\\
E^-(k)&=\bigcap\left\{E_{\psi'}(\T'_r): \vol(E_{\psi'}(\T'_r))\in [k^\alpha, (k+1)^\alpha]\right\}.
\end{split}\]
Since $\{E_{\psi'}(\T'_r)\}$ is increasing, we have that 
\begin{enumerate}[(1)]
\item $E^-(k)\subseteq E_{\psi'}(\T'_r)\subseteq E^+(k)$ when $\vol(E_{\psi'}(\T'_r))\in [k^\alpha, (k+1)^\alpha]$;
\item $k^\alpha \le \vol(E^-(k)) \le \vol(E^{+}(k)) \le (k+1)^\alpha$.
\end{enumerate}

We claim that 
\begin{equation}\label{discrete reduction 1}\begin{split}
&\limsup_{r}\left\{\Lambda \in Y_{\ov_d/N} : D(\Lambda, E_{\psi'}(\T'_r))> \vol(E_{\psi'}(\T'_r))^{\delta}\right\}\\
&\subseteq 
\limsup_{\scriptsize \begin{array}{c}
k\in \NN\\
\exists E^{\pm}(k)
\end{array}}
\left\{\Lambda \in Y_{\ov_d/N} : \begin{array}{c} D(\Lambda, E^+(k))> \frac 1 2 \vol(E^+(k))^{\delta}\;\text{or}\\[0.05in]
D(\Lambda, E^-(k))> \frac 1 2 \vol(E^-(k))^{\delta}\end{array}\right\}.
\end{split}\end{equation}
Indeed, for any $\Lambda\in Y_{\ov_d/N}$ such that there is $E_{\psi'}(\T'_r)$ with
\[
\vol(E_{\psi'}(\T'_r))\in [k^\alpha, (k+1)^\alpha]
\quad\text{and}\quad
D(\Lambda, E_{\psi'}(\T'_r))> \vol(E_{\psi'}(\T'_r))^\delta,
\]
using \Cref{discrepency}, since $3<\alpha < 1/(1-\delta)$, we have
\[\begin{split}
\max\left\{D(\Lambda, E^+(k)), D(\Lambda, E^-(k))\right\}
&> k^{\alpha \delta} - ck^{\alpha-1} > \frac 2 3 k^{\alpha\delta} \\
&> \frac 1 2 \max\{\vol(E^+(k)), \vol(E^-(k))\}^{\delta}
\end{split}\]
provided that $k\in \NN$ is sufficiently large, where
$c>0$ is taken so that $(k+1)^\alpha -k^\alpha < ck^{\alpha-1}$ for any $k\in \NN$.

By \Cref{GH22 Theorem 5.2} and the fact that $\alpha>1/(2\delta -1)$,
\[\begin{split}
&\sum_{\scriptsize \begin{array}{c} 
k\in \NN\\
\exists E^{\pm}(k)
\end{array}} 
\mu_S\left(\left\{\Lambda \in Y_{\ov_d/N} : \begin{array}{c} D(\Lambda, E^+(k))> \frac 1 2 \vol(E^+(k))^{\delta}\;\text{or}\\[0.05in]
D(\Lambda, E^-(k))> \frac 1 2 \vol(E^-(k))^{\delta}\end{array}\right\}
\right)\\
&\hspace{0.4in}\le
\sum_{\scriptsize \begin{array}{c} 
k\in \NN\\
\exists E^{\pm}(k)
\end{array}} 
4C_d \left(\vol(E^+(k))^{1-2\delta}+ \vol(E^-(k))^{1-2\delta}\right)\\
&\hspace{0.4in}\le
\sum_{k\in \NN} 8C_d (k+1)^{\alpha(1-2\delta)} <\infty.
\end{split}\]

Therefore, using Borel-Cantelli lemma and from \Cref{discrete reduction 1}, the set
$\limsup_{r}\left\{\Lambda \in Y_{\ov_d/N} : D(\Lambda, E_{\psi'}(\T'_r))> \vol(E_{\psi'}(\T'_r))^{\delta}\right\}$
is null, hence for almost all $\Lambda\in Y_{\ov_d/N}$, there is $r_0=r_0(\Lambda)$ such that
\[
\left|\#(\Lambda\cap E_{\psi'}(\T'_r))-\vol(E_{\psi'}(\T'_r))\right| 
< \vol(E_{\psi'}(\T'_r))^\delta. 
\]
for all $r\ge r_0$.
\end{proof}

Notice that in the proof of \Cref{Analogue of Schmidt 1}, to apply the Borel-Cantelli lemma, we needed an appropriate discretization on the set $\{\vol(E_{\psi'}(\T'_r))\}_r$ of volumes.
We want to use a similar technique to prove \Cref{Analogue of Schmidt 2}, and the condition for the set of volumes in the theorem is to ensure such a discretization.

\begin{proof}[Proof of \Cref{Analogue of Schmidt 2}]
Let such $(\delta_1, \delta, \alpha)$ and $C_{\psi'}>0$ exist. For each $k\in \NN$, define the set 
\[\begin{split}
\mathcal E^+(k)
&:=\left\{E_{\psi'}(\T'): \vol(E_{\psi'}(\T'))\in [k^\alpha, (k+1)^\alpha]
\;\text{and}\; E_{\psi'}(\T')\text{ is $(k,\alpha)$-maximal}\right\};\\
\mathcal E^-(k)
&:=\left\{E_{\psi'}(\T'): \vol(E_{\psi'}(\T'))\in [k^\alpha, (k+1)^\alpha]
\;\text{and}\; E_{\psi'}(\T')\text{ is $(k,\alpha)$-minimal}\right\}.
\end{split}\] 
It is obvious that for any $\T'\in \RR_{\ge 1}\times \prod_{p\in S_f} p^\NN$ with $\vol(E_{\psi'}(\T'))\ge 1$, there are $k\in \NN$ and sets $E^\pm \in \mathcal E^\pm(k)$ for which
\[
E^-\subseteq E_{\psi'}(\T') \subseteq E^+.
\]

Using the similar argument in the proof of \Cref{main thm 1}, since $\alpha<1/(1-\delta)$, we obtain that
\begin{equation}\label{discrete reduction 2}\begin{split}
&\limsup_{\T'}\left\{\Lambda \in Y_{\ov_d/N} : D(\Lambda, E_{\psi'}(\T'))> \vol(E_{\psi'}(\T'))^{\delta}\right\}\\
&\subseteq 
\limsup_{k\in \NN} 
\left\{\Lambda \in Y_{\ov_d/N} :  \begin{array}{c}
D(\Lambda, E)> \frac 1 2 \vol(E)^{\delta}\\[0.05in]
\text{for some }E\in \mathcal E^+(k)\cup\mathcal E^-(k)
\end{array}\right\}.
\end{split}\end{equation}

By the assumption, since $\#(\mathcal E^+(k) \cup \mathcal E^-(k))<C_{\psi'} k^{\delta_1}$, from \Cref{GH22 Theorem 5.2}, it follows that
\[\begin{split}
&\sum_{k\in \NN} \mu_S
\left\{\Lambda \in Y_{\ov_d/N} :  D(\Lambda, E)> \frac 1 2 \vol(E)^{\delta}\;\text{for some }E\in \mathcal E^+(k)\cup\mathcal E^-(k)\right\}\\
&\le \sum_{k\in \NN} C_{\psi'}k^{\delta_1}\times 4C_d (k+1)^{\alpha(1-2\delta)}<\infty
\end{split}\]
provided that $\alpha>(1+\delta_1)/(2\delta-1)$.
Again, the Borel-Cantelli lemma says that the limsup set in the left-hand side of \Cref{discrete reduction 2} is a null set and therefore we obtain the theorem.
\end{proof}

\subsection{Proof of \Cref{main thm 1} and \Cref{main thm 2}}
Let us prove \Cref{main thm 2}. 
Denote by
\[
\SU=\left\{\su_A=\left(\begin{array}{cc}
\id_m & A \\
0 & \id_n \end{array}\right): A\in \Mat_{m,n}(\QQ_S)\right\}
\]
and $\SH=\prod_{p\in S} H_p$, where for each $p\in S$,
\[
H_p=\left\{h_p=\left(\begin{array}{cc}
\alpha_p & 0\\
\beta_p & \gamma_p \end{array}\right)\in \SL_d(\QQ_p): 
\begin{array}{c}
\alpha_p\in \GL_m(\QQ_p), \; \gamma_p\in \GL_n(\QQ_p),\\
 \beta_p\in \Mat_{n,m}(\QQ_p)\end{array}\right\}.
\]
Then for any $\sh\in \SH$,
\[
\NT_{\psi,A}(\T)=\#\su_A\left(N\ZZ_S^d+(\ov_m,\ov_n)\right) \cap E_\psi(\T)=\#\sh\su_A\left(N\ZZ_S^d+\ov_d\right) \cap \sh E_{\psi}(\T).
\]

Following the tactic of \cite{AGY2021}, we will show that there is a sequence $(\sh_\ell)_{\ell\in \NN}$ in $\SH$ so that the following holds: 
for almost all $A\in \Mat_{m,n}(\QQ_S)$ and any small $\vep>0$, one can find $\ell\in \NN$ for which there is $\T_0=\T_0(A, \vep)$ ($=\T_0(A,\ell)$ to be exact) such that 
\begin{equation}\label{eq4: main thm}
\left|\frac {\NT_{\psi,A}(\T)} {\vol\left(E_{\psi}(\T)\right)/N^2}-1\right|
=\left|
\frac{\#\sh_\ell\su_A\left(N\ZZ_S^d+\ov_d\right) \cap \sh_{\ell}E_{\psi}(\T)}{\vol\left(E_{\psi}(\T)\right)/N^2} - 1\right|<\vep
\end{equation}
holds for all $\T\succeq \T_0$.

\vspace{0.1in}
Take any sequence $(\vep_\ell)_{\ell\in \NN}$ such that $\vep_\ell\in (0,1)$ and $\vep_\ell\rightarrow 0$ as $\ell\rightarrow \infty$. For each $\ell\in \NN$, let us define $\psi^+_{\ell}=(\psi^+_{p, \ell})_{p\in S}$ and $\psi^-_{\ell}=(\psi^-_{p, \ell})_{p\in S}$ by
\begin{equation}\label{psi pm}\begin{gathered}
\psi^+_{p,\ell}(\|\oy\|_p^n)=\left\{\begin{array}{cl}
(1+\vep_\ell) \psi_\infty\left(\dfrac 1 {1+\vep_\ell}\| \oy\|_\infty^n\right), &\text{for }p=\infty;\\[0.05in]
\psi_p(\|\oy\|^n_p), &\text{for }p\in S_f,
\end{array}\right.\quad\text{and}\\
\psi^-_{p,\ell}(\|\oy\|^n_p)=\left\{\begin{array}{cl}
\dfrac 1 {1+\vep_\ell} \psi_\infty\left((1+\vep_\ell)\| \oy\|_\infty^n\right), &\text{for }p=\infty;\\[0.15in]
\psi_p(\|\oy\|^n_p), &\text{for }p\in S_f.
\end{array}\right.
\end{gathered}\end{equation}

If $\psi$ has the condition in \Cref{main thm 2} (hence that of \Cref{Analogue of Schmidt 2}) with $(\delta_1, \alpha)$, then $\psi^{\pm}_\ell$ also has the same condition with the same pair $(\delta_1, \alpha)$ of constants (but possibly different $C_{\psi^\pm_\ell}>0$) for each $\ell\in \NN$.

\begin{lemma}\label{volume lemma}
Under the assumption of \Cref{Analogue of Schmidt 2},
for $(\vep_\ell)_{\ell\in \NN}$ and $(\psi^\pm_{\ell})_{\ell\in \NN}$ as above, one can find a sequence $(\sh_\ell)_{\ell\in \NN}$ in $\SH$ such that 
\[
E_{\psi^-_{\ell}}(\T^-)
\subseteq\sh_\ell E_\psi(\T)\subseteq E_{\psi^-_\ell}(\T^+),
\]
where $\T^-$ and $\T^+$ are defined as
$\T^-=\left(\frac 1 {1+\vep_\ell} T_\infty, T_{p_1}, \ldots, T_{p_s}\right)$ and
$\T^+=\left((1+\vep_\ell) T_\infty, T_{p_1}, \ldots, T_{p_s}\right)$, respectively when $\T=(T_p)_{p\in S}\in \RR_{\ge 1}\times \prod_{p\in S_f}\{p^k: k\in \NN\}$.

Moreover, it holds that for a.e. $A\in \Mat_{m,n}(\QQ_S)$ and any $\ell\in \NN$, 
\[
\# \sh_\ell \su_A\left(N\ZZ_S^d + \ov_d\right) \cap E_{\psi^{\pm}_{\ell}}(\T)=\frac {(1+\vep_\ell)^{\pm 2}} {N^2} \vol(E_{\psi}(\T))+ O_N(\vol(E_{\psi}(\T))^\delta)
\]
for all sufficiently large $\T$.
\end{lemma}
\begin{proof}
For each $\ell\in \NN$, define $H_{\infty, \vep_\ell}=\widetilde{H}_{\infty, \vep_\ell}\cap \widetilde{H}_{\infty, \vep_\ell}^{-1}$, where
\[
\widetilde{H}_{\infty, \vep_\ell}
=\left\{h_\infty=\left(\begin{array}{cc}
\alpha_\infty & 0 \\
\beta_\infty & \gamma_\infty\end{array}\right) \in H_\infty: 
\begin{array}{cc}
\|\alpha_\infty\|_{op}^m\le 1+\vep_\ell, \\[0.05in]
\|\gamma_\infty\|_{op}^n\le \frac {1+ {\vep_\ell}} 2,\;\|\beta_\infty\|_{op}\le \frac{\vep_\ell} {4n} \end{array}\right\}
\] 
and for $p\in S_f$,
\[
H_p(\ZZ_p)
=\left\{h_p=\left(\begin{array}{cc}
\alpha_p & 0 \\ 
\beta_p & \gamma_p \end{array}\right)\in H_p: 
\begin{array}{cc} 
\alpha_p\in \GL_m(\ZZ_p),\; \gamma_p\in \GL_n(\ZZ_p)\\
\beta_p\in \Mat_{n,m}(\ZZ_p)\end{array}\right\}.
\]
Here, $\|\cdot\|_{op}$ is the operator norm on $\RR^m$ or $\RR^n$ when $p=\infty$. It is well-known that $\SH_\ell:=H_{\infty, \vep_\ell}\times\prod_{p\in S_f} H_p(\ZZ_p)$ is an open neighborhood of the identity element in $\SH$.

Let us first show that for any $\sh\in \SH_\ell$ and $\T\in \RR_{\ge 1}\times \prod_{p\in S_f}\{p^k: k\in \NN\}$, 
\[
E_{\psi^-_{\ell}}(\T^-)
\subseteq\sh E_\psi(\T)\subseteq E_{\psi^+_\ell}(\T^+).
\]

If $p=\infty$, the sets
\[\begin{split}
&E_{\psi^\pm_{\infty,\ell}}((1+\vep_\ell)^{\pm 1}T_\infty)\\
&\hspace{0.3in}=\left\{(\ox,\oy)\in \RR^m\times \RR^n:\begin{array}{c}
\|\ox\|_\infty^m\le (1+\vep_\ell)^{\pm 1}\psi_\infty((1+\vep_\ell)^{\mp 1} \|\oy\|_\infty^n);\\[0.05in]
\|\oy\|_\infty^n \le (1+\vep_\ell)^{\pm 1} T_\infty
\end{array}\right\}.
\end{split}\]
Note that $h_\infty.(\ox,\oy)=(\alpha_\infty\ox, \beta_\infty\ox+\gamma_\infty\oy)$. 
For $(\ox,\oy)\in E_{\psi_\infty}(T_\infty)$,
\begin{equation}\label{eq3: main thm}\begin{split}
&\|\beta_\infty \ox+ \gamma_\infty \oy\|_\infty
\le \frac {\vep_\ell}{4n} \|\ox\|_\infty+ \left( 1+ \frac {\vep_\ell} 2 \right)^{1/n} \|\oy\|_\infty\\[0.1in]
&\hspace{0.2in}\le \frac {\vep_\ell} {4n} + \left( 1+ \frac {\vep_\ell} 2 \right)^{1/n} \|\oy\|_\infty
\le \left\{\begin{array}{cl}
 (1+\vep_\ell)^{1/n} &\text{if }\|\oy\|_\infty \le 1;\\[0.05in]
(1+\vep_\ell)^{1/n}\|\oy\|_\infty &\text{if }\|\oy\|_\infty \ge 1,
\end{array}\right.
\end{split}\end{equation}
and  
\[
\|\alpha_\infty \ox\|_\infty^m
\le (1+\vep_\ell)\psi_\infty(\|\oy\|_\infty^n)
\le (1+\vep_\ell)\psi_\infty((1+\vep_\ell)^{-1}\|\beta_\infty\ox+\gamma_\infty\oy\|_\infty^n),
\]
where the second inequality is induced from \Cref{eq3: main thm} since $\psi_\infty$ is non-increasing. 
Here, we use the fact that $\psi_\infty$ is a constant function on $(0,1]$ and $(1+\vep_\ell)^{-1}\|\beta_\infty+\gamma_\infty\oy\|_\infty^n\le 1$ if $\|\oy\|_\infty\le 1$.
This shows that $\sh E_{\psi_\infty}(T_\infty)\subseteq E_{\psi^+_{\infty,\ell}}((1+\vep_\ell)T)$ for $\sh\in \SH_\ell$.
One can also obtain the fact that $E_{\psi^-_{\infty,\ell}}((1+\vep_\ell)^{-1}T)\subseteq \sh E_{\psi_\infty}(T_\infty)$ in a similar way.

For $p\in S_f$, we claim that $H_p(\ZZ_p)$ preserves $E_{\psi_p}(T_p)$.
Again, the action of $h_p$ is given by $h_p.(\ox,\oy)=(\alpha_p\ox, \beta_p\ox+\gamma_p\oy)$. Note that $\alpha_p\in \GL_m(\ZZ_p)$ and $\gamma_p\in \GL_n(\ZZ_p)$ preserve the $p$-adic norm. Since $\|\ox\|_p\le |\psi_p|_{\sup}=1$ and $\beta_p\in \Mat_{n,m}(\ZZ_p)$,
\[
\|\beta_p\ox + \gamma_p\oy\|_p
\left\{\begin{array}{ll}
= \|\oy\|_p, &\text{if }\|\oy\|_p\ge p;\\
\le \max(\|\beta_p\ox\|_p, \|\gamma_p\oy\|_p)\le 1, &\text{if }\|\oy\|_p\le 1,
\end{array}\right.
\]
and it follows that
\[
\|\alpha_p\ox\|_p^m
=\|\ox\|_p^m
\le \psi_p(\|\oy\|_p^n)
=\left\{\begin{array}{cl}
\psi_p(\|\beta_p\ox+\gamma_p\oy\|_p^n), &\text{if }\|\oy\|_p\ge p;\\
1=\psi_p(\|\beta_p\ox+\gamma_p\oy\|_p^n), &\text{if }\|\oy\|_p\le 1.
\end{array}\right.
\]

\vspace{0.1in}
Now applying \Cref{Analogue of Schmidt 2} to $\psi^{\pm}_\ell$, it follows that for a.e. $\sg\in \SL_d(\QQ_S)$,
\begin{equation}\label{eq 1:pf of main thm} \begin{split}
&\#\sg\left(N\ZZ_S^d+\ov_d\right)\cap E_{\psi^{\pm}_\ell}(\T^\pm)\\
&\hspace{0.4in}=\frac 1 {N^2} \vol(E_{\psi^\pm_\ell}(\T^\pm))+O_N(\vol(E_{\psi^\pm_\ell}(\T^\pm))^\delta)
\end{split}\end{equation}
for all sufficiently large $\T^\pm$. Moreover, it is easy to compute that the volumes of $E_{\psi^\pm_\ell}(\T^\pm)$ are 
\begin{equation}\label{eq 2:pf of main thm}
\vol(E_{\psi^\pm_\ell}(\T^\pm))=(1+\vep_\ell)^{\pm 2} \vol\left(E_{\psi}(\T)\right).
\end{equation}

Consider the map $\phi:\SH \times \SU \rightarrow \SL_d(\QQ_S)$ given by $\phi(\sh,\su_A)=\sh\su_A$ which is a diffeomorphism up to a null set in $\SL_d(\QQ_S)$. Then $(\phi^{-1})_*(\mu_S)$ is equivalent to $\mu_\SH \otimes \mu_\SU$, where $\mu_\SH$ is a Haar measure on $\SH$ and $\mu_\SU=\vol$ under the natural identification $\SU\simeq \Mat_{m,n}(\QQ_S)$.
Hence from \Cref{eq 1:pf of main thm}, we have that
\[
\#\sh\su_A\left(N\ZZ_S^d+\ov_d\right)\cap E_{\psi^{\pm}_\ell}(\T^\pm)
=\frac 1 {N^2} \vol(E_{\psi^\pm_\ell}(\T^\pm))+O_N(\vol(E_{\psi^\pm_\ell}(\T^\pm))^\delta)
\]
for a.e. $(\sh,\su_A) \in \SH\times \SU$ with respect to the product measure $\mu_\SH\otimes \mu_\SU$. 
In particular, we can choose $\sh_\ell\in \SH_\ell$ for each $\ell\in\NN$ such that the above equation holds for a.e. $\su_A$. Let $\SU_\ell$ be a collection of such elements $\su_A\in \SU$ with respect to $\sh=\sh_\ell$. Then $\SU_\ell$ is of full measure, hence the intersection $\bigcap_{\ell\in \NN} \SU_\ell$ also has a full measure. Combining with \Cref{eq 2:pf of main thm}, the second assertion of the lemma follows.
\end{proof}

Now let us finish the proof of \Cref{main thm 2} by showing \Cref{eq4: main thm} for those $A\in \Mat_{m,n}(\QQ_S)$ as in \Cref{volume lemma}. 

Let sufficiently small $\vep>0$ be given. Choose $\ell\in \NN$ so that $\vep_\ell < \vep/6$. For $\sh_\ell$ in \Cref{volume lemma}, we have that
\[
\frac{\#\sh_\ell \su_A(N\ZZ_S^d+\ov_d)\cap E_{\psi^-_{\ell}}(\T)}{\vol(E_\psi(\T))/N^2}
\le \frac{N_{\psi, A}(\T)}{\vol(E_\psi(\T))/N^2}
\le \frac{\#\sh_\ell \su_A(N\ZZ_S^d+\ov_d)\cap E_{\psi^+_\ell}(\T)}{\vol(E_\psi(\T))/N^2}
\]
and
\[
\#\sh_\ell \su_A(N\ZZ_S^d+\ov_d)\cap E_{\psi^\pm_\ell}(\T)
=(1+\vep_\ell)^{\pm 2} \frac {\vol(E_\psi(\T))}{N^2} + O_N(\vol(E_\psi(\T))^\delta).
\]

Hence it follows that
\[
\left|\frac{N_{\psi, A}(\T)}{\vol(E_\psi(\T))/N^2}-1\right|
\le 3\vep_\ell + O_N(\vol(E_\psi(\T))^{\delta-1})
\le \vep
\]
for all sufficiently large $\T$ since we assume that $\vol(E_\psi(\T))\rightarrow \infty$ as $\T\rightarrow \infty$.

\vspace{0.1in}
The proof of \Cref{main thm 1} is similarly obtained when we use the sequence $(\T_r)_r$ in \Cref{main thm 1} instead of $\T\in \RR_{\ge 1}\times \prod_{p\in S_f} p^{\NN\cup\{0\}}$, and \Cref{volume lemma 1} using \Cref{Analogue of Schmidt 1}.

\begin{lemma}\label{volume lemma 1}
Under the assumption of \Cref{Analogue of Schmidt 1},
for $(\vep_\ell)_{\ell\in \NN}$ and $(\psi^\pm_{\ell})_{\ell\in \NN}$ as in \Cref{psi pm}, one can find a sequence $(\sh_\ell)_{\ell\in \NN}$ in $\SH$ such that 
\[
E_{\psi^-_{\ell}}(\T_r^-)
\subseteq\sh_\ell E_\psi(\T_r)\subseteq E_{\psi^-_\ell}(\T_r^+),\; \forall r,
\]
where $\T_r^\pm$ is defined as in \Cref{volume lemma}.
Moreover, it holds that for a.e. $A\in \Mat_{m,n}(\QQ_S)$ and any $\ell\in \NN$, 
\[
\# \sh_\ell \su_A\left(N\ZZ_S^d + \ov_d\right) \cap E_{\psi^{\pm}_{\ell}}(\T_r)=\frac {(1+\vep_\ell)^{\pm 2}} {N^2} \vol(E_{\psi}(\T_r))+ O_N(\vol(E_{\psi}(\T_r))^\delta)
\]
for all sufficiently large $r$.
\end{lemma}

\section{Further Questions}

\subsection{General divergent collection of approximating functions}\label{Dirichlet case} The restriction of the growth of the number of sets $E_\psi(\T)$ in \Cref{main thm 2} is crucial to use the Borel-Cantelli lemma in our proof. Even the collection of approximating function given as in \Cref{New Dirichlet cor} does not satify the condition: Suppose that $S=\{\infty, p\}$, and $\psi=(\psi_p)_{p\in S}$ is given by
\[
\psi_\infty(T_\infty)=\min\left\{1, \frac 1 {T_\infty}\right\}
\quad\text{and}\quad
\psi_p(T_p)\asymp \min\left\{1, \frac 1 {T_p}\right\}
\]
so that $\VT(\T)\asymp \log T_\infty \times \log_p T_p$.
Then for any $\alpha>0$, 
\[\begin{split}
&\#\left\{\T\in \RR_{\ge 1}\times p^{\NN\cup\{0\}}:
\begin{array}{c} 
\VT(\T)/N^d\in [k^\alpha, (k+1)^\alpha];\\
\T\text{ is $(N,k,\alpha)$-maximal or minimal}\\
\end{array} \right\}\\
&\hspace{0.4in}=\#\left\{\T\in \RR_{\ge 1}\times p^{\NN\cup\{0\}}:\VT(\T)=k^\alpha\;\text{or}\;(k+1)^\alpha\right\}\\
&\hspace{0.4in}\asymp\#\left\{T_p=p^{t_p}\in p^{\NN\cup\{0\}}: 1\ll t_p \ll k^\alpha\right\} 
\asymp k^\alpha,\;\forall k\in \NN.
\end{split}\]
Hence $\delta_1\ge\alpha$, where $(\delta_1, \delta, \alpha)$ is as in \Cref{main thm 2}, then the condition that 
\[
\alpha> \frac {1+\delta_1}{2\delta-1}\ge\frac {1+\alpha} {2\delta-1}
\quad\Rightarrow\quad
\delta> \frac {1+2\alpha} {2\alpha}
\]
contradicts to the condition that $\delta<1$.

\subsection{Counting solutions $(\op,\oq)\in \mathcal R^m\times \mathcal R^n$ for a subring $\mathcal R\subseteq\ZZ_S$}
It is well-known that a subring of $\ZZ_S$ is of the form $\ZZ_{S'}=\ZZ[1/{\prod_{p\in S'} p}]$ (as a subring of $\QQ$), where $S' \subseteq S$ (when $S'=\{\infty\}$, $\mathcal R=\ZZ$).
Then one can consider the Khintchine-Groshev type theorem for the set of $A=(A_p)_{p\in S}\in \Mat_{m,n}(\QQ_S)$ such that there are infinitely many $(\op,\oq)\in \mathcal R^m\times \mathcal R^n$ for which
\[
\|A_p\oq+\op\|_p^m \le \psi_p(\|\oq\|_p^n),\;\forall p\in S,
\]
where $\psi=(\psi_p)_{p\in S}$ is a collection of approximating functions.

However, when we want to quantify this theorem for the case when $\mathcal R\subsetneq \ZZ_S$, since $\mathcal R^d$ and $\SL_d(\mathcal R)$ are not lattice subgroups of $\QQ_S^d$ or $\SL_d(\QQ_S)$ respectively, there are no moment formulas (which are so-called \emph{Rogers' formulas}) crucially used in proving \Cref{GH22 Theorem 5.2}.

In this point of view, it is also difficult to consider the function counting $(\op,\oq)\in \ZZ^m\times \ZZ^n$ with inequalities given by the collection $\psi=(\psi_p)_{p\in S}$ of approximating functions, where $S$ is the set of finite places only, which is the case similar to that of \cite{KT2007} (more precisely, the authors in \cite{KT2007} concerned a Khintchine-Groshev type theorem using the $S$-norm with $S$ not containing the infinite place) by using the method presented in this article. 

\subsection{Quantitative Khintchine-Groshev theorem for a different $S$-arithmetic generalization}
As described in \Cref{Sec: Divergence Case: The Quantitative Result},
one can quantify the Khintchine-Groshev theorem which counts
\[
\#\left\{(\op,\oq)\in \ZZ_S^m\times \ZZ_S^n: 
\|\oq\|_S^n \le T\;\text{and}\; \|A\oq-\op\|_S^m < \psi_0(\|\oq\|_S^n) \right\}
\]
as $T$ goes to infinity, for a given non-increasing function $\psi_0:\RR_{>0}\rightarrow\RR_{>0}$. 
Unlike to two $S$-arithmetic analogs of Dirichlet's theorem and corollary, \Cref{New Khintchine-Groshev thm} does not imply the $S$-arithmetic Khintchine-Groshev theorem introduced in \cite{KT2007} and vice versa.
To the best of my knowledge, it is hard to obtain a statement similar to \Cref{GH22 Theorem 5.2} even for an approximating function of the form $\psi_0(\|\oq\|_S^n)\asymp \|\oq\|_S^{-\nu}$ for some positive number $\nu>0$, which is one of key steps in our strategy.


\end{document}